\documentclass{amsart}
\usepackage{amsmath}
\usepackage{amscd}
\usepackage{amssymb}
\usepackage{amsthm}
\usepackage{color}
\usepackage{tcolorbox}
\RequirePackage{filecontents}
\usepackage{fullpage}
\usepackage[numbers]{natbib}  
\usepackage[draft]{hyperref}
\usepackage{diagbox}

\usepackage{tikz}
\usetikzlibrary{positioning}
\usetikzlibrary{arrows}

\theoremstyle{plain}
\newtheorem{prop}{Proposition}
\newtheorem{lem}[prop]{Lemma}

\theoremstyle{definition}

\newtheorem{ex}[prop]{Example}

\newenvironment{psmallmatrix}
  {\left(\begin{smallmatrix}}
  {\end{smallmatrix}\right)}
  
\newcommand{\sig}{\mathrm{sig}}

\newcommand{\Z}{\mathbb{Z}}
\newcommand{\Q}{\mathbb{Q}}
\newcommand{\C}{\mathbb{C}}
\renewcommand{\H}{\mathbb{H}}
\newcommand{\SL}{\mathrm{SL}}
\newcommand{\Mp}{\mathrm{Mp}}

\newcommand{\e}{\mathfrak{e}}
\newcommand{\AL}{\mathrm{AL}}
\newcommand{\even}{\mathrm{even}}
\newcommand{\odd}{\mathrm{odd}}

\begin{document}

\title{Twisted component sums of vector-valued modular forms}

\author{Markus Schwagenscheidt}
\address{Mathematical Institute, University of Cologne, Weyertal 86-90, 50931 Cologne, Germany}
\email{mschwage@math.uni-koeln.de}

\author{Brandon Williams}
\address{Fachbereich Mathematik \\ Technische Universit\"at Darmstadt \\ 64289 Darmstadt, Germany}
\email{bwilliams@mathematik.tu-darmstadt.de}


\thanks{We thank Jan H. Bruinier and Stephan Ehlen for helpful discussions. M. Schwagenscheidt is supported by the SFB-TRR 191 \lq Symplectic Structures in Geometry, Algebra and Dynamics\rq, funded by the DFG. B. Williams is supported by the LOEWE research unit Uniformized Structures in Arithmetic and Geometry.}

\maketitle

\begin{abstract}
	We construct isomorphisms between spaces of vector-valued modular forms for the dual Weil representation and certain spaces of scalar-valued modular forms in the case that the underlying finite quadratic module $A$ has order $p$ or $2p$, where $p$ is an odd prime. The isomorphisms are given by twisted sums of the components of vector-valued modular forms. Our results generalize work of Bruinier and Bundschuh to the case that the components $F_{\gamma}$ of the vector-valued modular form are antisymmetric in the sense that $F_{\gamma} = -F_{-\gamma}$ for all $\gamma \in A$. As an application, we compute restrictions of Doi-Naganuma lifts of odd weight to components of Hirzebruch-Zagier curves.
\end{abstract}

\section{Introduction}

In the study of theta lifts (such as Maass lifts and Borcherds products) it is convenient to work with vector-valued modular forms for the dual Weil representation $\rho^{*}$ associated to a finite quadratic module $(A,Q)$. Therefore, it is useful to understand the precise relationship between vector-valued modular forms for $\rho^{*}$ and scalar-valued modular forms for congruence subgroups. 

For example, in some cases there are isomorphisms between spaces of vector-valued and scalar-valued modular forms. In \cite{BB}, Bruinier and Bundschuh showed that if $|A| = p$ is an odd prime, then modular forms for $\rho^{*}$ of weight $k \in \Z$ with $k \equiv \sig(A,Q)/2 \pmod 2$ can be identified with certain modular forms of weight $k$ for $\Gamma_{0}(p)$ and Nebentypus $\chi_{p} = \big( \frac{\cdot}{p}\big)$. The isomorphism is given by the component sum
\[
\varphi\left(\sum_{\gamma \in A}F_{\gamma}(\tau)\e_{\gamma}\right) = \sum_{\gamma \in A}F_{\gamma}(p\tau)
\]
of the vector-valued modular form $F(\tau) =\sum_{\gamma \in A}F_{\gamma}(\tau)\e_{\gamma}$, where $\e_{\gamma}$ denotes the standard basis of the group algebra $\C[A]$. Using similar ideas, Y. Zhang constructed isomorphisms between spaces of vector-valued and scalar-valued modular forms for certain classes of finite quadratic modules which do not necessarily have odd prime order (see \cite{Z1,Z2}).

The condition $k\equiv \sig(A,Q)/2 \pmod 2$ turns out to be crucial for the aforementioned result of Bruinier and Bundschuh, since otherwise the components of any modular form for $\rho^{*}$ satisfy $F_{-\gamma} = -F_{\gamma}$ and hence cancel out in pairs in the sum. To obtain a non-zero map in any weight, we twist the component sums of vector-valued modular forms by a Dirichlet character $\chi$ mod $p$ with $\chi(-1) = (-1)^{k+\sig(A,Q)/2}$. Suppose that $|A| = p$ with an odd prime $p$. We define the \textbf{twisted component sum} of a modular form $F(\tau) = \sum_{\gamma \in A}F_{\gamma}(\tau)\e_{\gamma}$ for $\rho^{*}$ by 
\[
\varphi_{\chi}\left(\sum_{\gamma \in A}F_{\gamma}(\tau)\e_{\gamma} \right) = \sum_{\gamma \in A}\chi(\gamma)F_{\gamma}(p\tau),
\]
where we fix an identification of $A$ with $\Z/p\Z$ to define $\chi(\gamma)$ for $\gamma \in A$. The assumptions on $k$ and $\chi$ imply that $\varphi_{\chi}$ is not trivially zero. Moreover, we have the following result.

\begin{prop}
	The map $\varphi_{\chi}$ defines an injective homomorphism from the space of modular forms of weight $k$ for $\rho^{*}$ to the space of scalar-valued modular forms of weight $k$ for $\Gamma_{0}(p^{2})$ with Nebentypus $\chi\otimes \chi_{p}$.
\end{prop}

For the proof we refer to Proposition~\ref{prop varphi injective} below. It is immediate from the construction that the $n$-th Fourier coefficient of $\varphi_{\chi}(F)$ vanishes unless $n \in p(\Z-Q(\gamma))$ for some $\gamma \in A \setminus \{0\}$. However, $\varphi_{\chi}$ is in general not surjective onto the subspace defined by this vanishing condition. We characterize the image of $\varphi_{\chi}$ in terms of the Atkin-Lehner involution in Proposition~\ref{prop isomorphism p} below. 

We construct an analogous map $\varphi_{\chi}$ in the case that $|A| = 2p$ is twice an odd prime $p$, see Proposition~\ref{prop isomorphism 2p}; in this case, $(A,Q)$ must have odd signature and all modular forms are of half-integral weight.

As an application, we compute restrictions of Doi-Naganuma lifts of odd weight to components of Hirzebruch-Zagier curves $T_{\ell}$ of prime index $\ell$. Let $K = \Q(\sqrt{p})$ with a prime $p \equiv 1 \pmod 4$ and let $\mathcal{O}_{K}$ be its ring of integers. Recall that the Doi-Naganuma lift maps a vector-valued cusp form $F$ of weight $k$ for the dual Weil representation associated to the lattice $(\mathcal{O}_{K},-N_{K/\Q})$ to a Hilbert cusp form $\Phi_{F}$ of weight $k$ for $\SL_{2}(\mathcal{O}_{K})$. The restriction of $\Phi_{F}$ to a component of the Hirzebruch-Zagier curve $T_{\ell}$ of prime index $\ell$ is given by the Shimura lift of the vector-valued cusp form of weight $k+1/2$ for the dual Weil representation of the lattice $(\Z,-\ell x^{2})$ obtained by the so-called theta contraction of $F$ as defined in \cite{M}, i.e. we have the commutative diagram
\begin{center} \begin{tikzpicture}[node distance=2.5cm, auto]
\node (A) {$S_{k}(-N_{K/\Q})$};
\node (AA) [right of=A] {$ $};
\node (B) [right of=AA] {$S_k(\SL_{2}(\mathcal{O}_{K}))$};
\node (C) [below of=A] {$S_{k+1/2}((-2\ell))$};
\node (D) [below of=B] {$S_{2k}(\Gamma_{0}(\ell))$};
\draw[->] (B) to node {\text{restr. to $T_{\ell}$}} (D);
\draw[->] (A) to node {\text{theta contr.}} (C);
\draw[->] (A) to node {$\mathrm{Doi-Naganuma}$} (B);
\draw[->] (C) to node {$\mathrm{Shimura}$} (D);
\end{tikzpicture}
\end{center}
For a proof, see Lemma~\ref{lem doi-naganuma-shimura}. We show that, on the level of the corresponding scalar-valued modular forms, the theta contraction basically becomes a multiplication by the Jacobi theta function. In this way, passing to scalar-valued modular forms makes it easier to compute the restriction of $\Phi_{F}$ to a component of $T_{\ell}$. To illustrate the result, we consider the case $\ell = p$ in the introduction.

\begin{prop}
Let $\chi$ be a Dirichlet character mod $p$ with $\chi(-1) = (-1)^{k}$. We have the following commutative diagram:
\begin{center} \begin{tikzpicture}[node distance=2.5cm, auto]
\node (A) {$S_{k}(-N_{K/\Q})$};
\node (AA) [right of=A] {$ $};
\node (B) [right of=AA] {$S_k(\Gamma_{0}(p^{2}),\chi\otimes \chi_{p})$};
\node (C) [below of=A] {$S_{k+1/2}((-2\ell))$};
\node (D) [below of=B] {$S_{k+1/2}(\Gamma_{0}(p^{2}),\chi)$};
\draw[->] (B) to node {$G \mapsto \big(G(4p\tau)\cdot \vartheta(\tau)\big)|U_{p}$} (D);
\draw[->] (A) to node {$\mathrm{theta}\, \mathrm{contr.}$} (C);
\draw[->] (A) to node {$\varphi_{\chi}$} (B);
\draw[->] (C) to node {$\varphi_{\chi}$} (D);
\end{tikzpicture}
\end{center}
where $\vartheta(\tau) = \sum_{n \in \Z}q^{n^{2}}$ is the Jacobi theta function and $U_{p}$ is the usual Hecke operator acting on Fourier expansions by $\big(\sum_{n}c(n)q^{n}\big)|U_{p} = \sum_{n}c(pn)q^{n}$.
\end{prop}

We refer to Proposition~\ref{prop theta contraction} for the general statement and its proof. We also give two numerical examples illustrating the use of the above proposition in Section~\ref{sec Doi-Naganuma}.

The work is organized as follows. We start with preliminaries about modular forms for the Weil representation associated to a finite quadratic module. In Section~\ref{sec isomorphisms}, we investigate twisted component sums of vector-valued modular forms and obtain isomorphisms between spaces of vector-valued and scalar-valued modular forms in the case that the underlying finite quadratic module has order $p$ or $2p$, with an odd prime $p$. Finally, in Section~\ref{sec Doi-Naganuma}, we explain how these isomorphisms can be used to compute restrictions of Doi-Naganuma lifts of odd weight to components of Hirzebruch-Zagier curves.

\section{Modular forms for the Weil representation}

A \textbf{finite quadratic module} $(A,Q)$ consists of a finite abelian group $A$ and a nondegenerate $\Q/\Z$-valued quadratic form $Q$ on it. The \textbf{signature} of $(A,Q)$ is the number $\sig(A,Q) \in \Z/8\Z$ defined through the Gauss sum of $A$ by 
\begin{align}\label{eq Milgram formula}
\mathbf{e}(\sig(A,Q)/8) = \frac{1}{\sqrt{|A|}} \sum_{\gamma \in A} \mathbf{e}(Q(\gamma)),
\end{align}
where $\mathbf{e}(x) = e^{2\pi i x}$. By Milgram's formula (\cite{MH}, appendix 4), this is also the signature mod $8$ of any even lattice which induces $(A,Q)$ as its discriminant form.

Let $\C[A]$ be the group algebra of $A$ with basis $\e_{\gamma}$, $\gamma \in A,$ and let $\Mp_{2}(\Z)$ be the integral metaplectic group, consisting of pairs $\left(\left(\begin{smallmatrix}a & b \\ c & d \end{smallmatrix}\right),\pm \sqrt{c\tau + d}\right)$ with $\left(\begin{smallmatrix}a & b \\ c & d \end{smallmatrix}\right) \in \SL_{2}(\Z)$. The \textbf{dual Weil representation} $\rho^{*}$ is a unitary representation of $\Mp_{2}(\Z)$ on $\C[A]$ which is defined on the generators $S = \left(\begin{psmallmatrix} 0 & -1 \\ 1 & 0 \end{psmallmatrix},\sqrt{\tau}\right)$ and $T =\left( \begin{psmallmatrix} 1 & 1 \\ 0 & 1 \end{psmallmatrix},1\right)$ by 
\[
\rho^*(T) \e_{\gamma} = \mathbf{e}(-Q(\gamma)) \e_{\gamma}, \qquad 
\rho^*(S) \e_{\gamma} = \frac{\mathbf{e}(\sig(A,Q)/8) }{\sqrt{|A|}} \sum_{\beta \in A} \mathbf{e}(\langle \beta, \gamma \rangle) \e_{\beta},
\]
where $\langle \beta, \gamma \rangle = Q(\beta + \gamma) - Q(\beta) - Q(\gamma)$ is the bilinear form associated to $Q$. We also write $\rho^{*}_{A}$ if we want to emphasize the dependence on $A$, or $\rho^{*}_{\Lambda}$ if $A$ is the discriminant form of an even lattice $\Lambda$.

A function $F: \H \to \mathbb{C}[A]$ is called a \textbf{weakly holomorphic modular form} of weight $k \in \frac{1}{2}\Z$ for $\rho^{*}$ if it is holomorphic on $\H$, if it satisfies 
\[
F\left( \frac{a\tau+b}{c\tau+d}\right) = (c\tau+d)^{k}\rho^{*}\left(\begin{pmatrix}a & b \\ c & d \end{pmatrix},\sqrt{c\tau+d}\right)F(\tau)
\]
for all $\left(\left(\begin{smallmatrix}a & b \\ c & d \end{smallmatrix}\right),\sqrt{c\tau+d}\right) \in \Mp_{2}(\Z)$, and if it is meromorphic at $\infty$, which means that it has a Fourier expansion of the form
\[
F(\tau) = \sum_{\gamma \in A}\sum_{\substack{n \in \Z-Q(\gamma) \\ n \gg -\infty}}c(n,\gamma)q^{n}\e_{\gamma},
\] 
with coefficients $c(n,\gamma) \in \C$ and $q = e^{2\pi i \tau}$. Following \cite{BB}, we will denote the space of all these functions by $A_{k}(\rho^{*})$ (instead of the more common $M_{k}^{!}(\rho^{*})$). We let $M_{k}(\rho^{*})$ and $S_{k}(\rho^{*})$ be the subspaces of holomorphic modular forms and cusp forms, respectively. If $A$ is the discriminant form of an even lattice $\Lambda$, then we also write $A_{k}(Q)$ or $A_{k}(\mathbf{S})$ for $A_{k}(\rho^{*})$, where $Q$ is the quadratic form on $\Lambda$ and $\mathbf{S}$ is the Gram matrix of $Q$ with respect to some basis of $\Lambda$.

The element $Z = (-I,i) = S^{2}$ acts by $\rho^{*}(Z)\e_{\gamma} = (-1)^{\sig(A,Q)/2}\e_{-\gamma}$ which implies that $A_{k}(\rho^{*}) = 0$ if $k + \sig(A,Q)/2$ is not integral, and that the components of any weakly holomorphic modular form $F = \sum_{\gamma \in A}F_{\gamma}\e_{\gamma} \in A_{k}(\rho^{*})$ satisfy
\[
F_{\gamma} = (-1)^{k+\sig(A,Q)/2}F_{-\gamma}
\]
for all $\gamma \in A$. Therefore we refer to $k$ as a \textbf{symmetric} or \textbf{antisymmetric weight} if $k+\sig(A,Q)/2$ is respectively even or odd.

\section{Vector-valued and scalar-valued modular forms}\label{sec isomorphisms}

In this section, we give isomorphisms between spaces $A_{k}(\rho^{*})$ of vector-valued modular forms for $\rho^{*}$ and scalar-valued modular forms for $\Gamma_{0}(p^{2})$ and $\Gamma_{0}(4p^{2})$ in the cases $|A| = p$ and $|A| = 2p$ with an odd prime $p$, for both symmetric and antisymmetric weights $k \in \frac{1}{2}\Z$.

\subsection{Finite quadratic modules of order $p$} Suppose that $|A| = p$ is an odd prime. Then $A \cong \Z/p\Z$ with $Q(\gamma) = \alpha\gamma^{2}/p$ for some $\alpha \in \Z$ with $p \nmid \alpha$. We put $\epsilon = \chi_{p}(\alpha) = \big( \frac{\alpha}{p}\big)$, and for odd $d \in \Z$ we let 
\[ 
\varepsilon_{d} = 
	\begin{cases}
		1, & p \equiv 1 \pmod 4, \\ 
		i, & p\equiv 3 \pmod 4.
	\end{cases}
\] 
The evaluation of the quadratic Gauss sum $\sum_{n(p)}\mathbf{e}(\alpha n^{2}/p) = \chi_{p}(\alpha)\varepsilon_{p}\sqrt{p}$ and Milgram's formula \eqref{eq Milgram formula} show that $\epsilon\varepsilon_{p} = \mathbf{e}(\sig(A,Q)/8)$. Thus the signature $\sig(A,Q) \in \Z/8\Z$ depends on $p$ and $\epsilon$ as shown in the following table:
\begin{align*}
	\begin{array}{c||c|c}
	p \pmod 4 & 1 & 3 \\
	\hline\hline
	\epsilon =+1 & 0 & 2\\
	\hline
	\epsilon =-1 & 4 & 6 \\
	\end{array}
\end{align*}
In particular, $\sig(A,Q)$ is even. Hence we can assume that $k$ is an integer since otherwise $A_{k}(\rho^{*}) = 0$.

Let $\chi$ be a Dirichlet character mod $p$ and let $A_{k}(p^{2},\chi\otimes \chi_{p})$ be the space of scalar-valued weakly holomorphic modular forms of weight $k$ for $\Gamma_{0}(p^{2})$ with character $\chi\otimes \chi_{p}$. We assume that 
\begin{align}\label{eq parity chi}
\chi(-1) = (-1)^{k+\sig(A,Q)/2}
\end{align}
since otherwise $A_{k}(p^{2},\chi\otimes\chi_{p}) = 0$. We define the subspace
\begin{align*}
A_{k}^{\epsilon}(p^{2},\chi\otimes\chi_{p}) 
= \left\{\sum_{\substack{n \in \Z \\ n\gg -\infty}}c(n)q^{n} \in A_{k}(p^{2},\chi\otimes\chi_{p})\ : \ \begin{aligned} & c(n) = 0 \text{ unless } n \in p(\Z-Q(\gamma)) \\ &\text{for some } \gamma \in (\Z/p\Z)^{*} \end{aligned}\right\}.
\end{align*}
The condition in the brackets can be restated by saying that $c(n) = 0$ unless $\chi_{p}(-n)=\epsilon$. Hence we also call it the \textbf{$\epsilon$-condition}. Note that, in contrast to the definition of the $\epsilon$-condition in \cite{BB}, we also require that $c(n) = 0$ if $p \mid n$. 

We define the \textbf{twisted component sum} of a vector-valued modular form $F(\tau) = \sum_{\gamma (p)}F_{\gamma}(\tau)\e_{\gamma} \in A_{k}(\rho^{*})$ by 
\[
\varphi_{\chi}\left( \sum_{\gamma (p)}F_{\gamma}(\tau)\e_{\gamma}\right) = \sum_{\gamma (p)^{*}}\chi(\gamma)F_{\gamma}(p\tau).
\]
The parity condition \eqref{eq parity chi} ensures that $\varphi_{\chi}(F)$ is not trivially identically zero. Again, also for symmetric weight $k$ and with $\chi$ the trivial character mod $p$, our twisted component sum differs from the component sum $\sum_{\gamma(p)}F_{\gamma}(p\tau)$ considered in \cite{BB} since we omit the zero component $F_{0}(p\tau)$ in the sum. For antisymmetric weight $k$ we have $F_{0} = 0$.

\begin{prop}\label{prop varphi injective}
	If $F \in A_{k}(\rho^{*})$, then $\varphi_{\chi}(F) \in A_{k}^{\epsilon}(p^{2},\chi\otimes\chi_{p})$. Furthermore, $\varphi_{\chi}$ is injective.
\end{prop}

\begin{proof}
	Let $M = \left(\begin{smallmatrix}a & b \\ c & d \end{smallmatrix} \right) \in \Gamma_{0}(p^{2})$ and $N = \left(\begin{smallmatrix}a & bp \\ c/p & d \end{smallmatrix} \right)$. We write
	\begin{align*}
	\varphi_{\chi}(F)(M\tau) = \sum_{\gamma (p)^{*}}\chi(\gamma)F_{\gamma}(p\cdot M\tau) = \sum_{\gamma (p)^{*}}\chi(\gamma)F_{\gamma}(N(p\tau)).
	\end{align*}
	By \cite{BReflection}, Theorem 5.2, we have $\rho^{*}(N)\e_{\gamma} = \chi_{p}(d)\e_{d\gamma}$, hence
	\[
	F_{\gamma}(N(p\tau)) =\chi_{p}(d) (c\tau + d)^{k}F_{d^{-1}\gamma}(p\tau),
	\]
	where $d^{-1}$ denotes an inverse of $d$ mod $p$. Thus we find
	\begin{align*}
	\varphi_{\chi}(F)(M\tau) &= \chi_{p}(d)(c\tau + d)^{k}\sum_{\gamma (p)^{*}}\chi(\gamma)F_{d^{-1}\gamma}(p\tau) = \chi\chi_{p}(d)(c\tau+d)^{k}\varphi_{\chi}(F)(\tau).
	\end{align*}
	It is clear that $\varphi_{\chi}(F)$ is holomorphic on $\H$ and meromorphic at the cusps (since $F_{\gamma}|_{k}M$ is a linear combination of components of $F$), and that it satisfies the $\epsilon$-condition, so $\varphi_{\chi}(F) \in A_{k}^{\epsilon}(p^{2},\chi\otimes\chi_{p})$.

	Now suppose that $\varphi_{\chi}(F) = 0$ for some $F \in A_{k}(\rho^{*})$. Since the components $F_{\gamma},F_{\beta}$ for $\beta \neq \pm \gamma$ are supported on disjoint index sets, $\varphi_{\chi}(F) = 0$ implies that $F_{\gamma} = 0$ for all $\gamma \neq 0$. But then the zero component of $F$ satisfies 
	\[
	F_{0}|_{k}S = \frac{\mathbf{e}(\sig(A,Q)/8)}{\sqrt{p}}\sum_{\gamma(p)}F_{\gamma} = \frac{\mathbf{e}(\sig(A,Q)/8)}{\sqrt{p}}F_{0},
	\]
	and applying $|_{k}S$ a second time, we find $(-1)^{k}F_{0} = \frac{\mathbf{e}(\sig(A,Q)/4)}{p}F_{0}$, hence $F_{0} = 0$. We have shown that $F = 0$, so $\varphi_{\chi}$ is injective.
\end{proof}

The map $\varphi_{\chi}$ is in general not surjective. Its image can be described in terms of the behaviour of certain twists of $G$ under the Atkin-Lehner involution, which we explain now.

We can split $G(\tau) = \sum_{n \gg -\infty}c(n)q^{n} \in A_{k}^{\epsilon}(p^{2},\chi\otimes\chi_{p})$ into components
\[
G(\tau) = \sum_{\gamma(p)^{*}}G_{\gamma}(\tau), \qquad G_{\gamma}(\tau) = \frac{1}{2}\sum_{\substack{n \in p(\Z-Q(\gamma)) \\ n\gg -\infty}}c(n)q^{n}.
\]
We define the \textbf{component-wise twist} of $G$ by a Dirichlet character $\psi$ mod $p$ by
\[
G_{\psi}(\tau) = \sum_{\gamma (p)^{*}}\psi(\gamma)G_{\gamma}(\tau).
\]
Note that the component-wise twist differs from the usual twist $\sum_{n \gg-\infty}\psi(n)c(n)q^{n}$ of a modular form.

\begin{lem}\label{lem twisting}
	Let $G \in A_{k}^{\epsilon}(p^{2},\chi\otimes \chi_{p})$ and let $\psi$ be a Dirichlet character mod $p$. Then $G_{\psi} \in A_{k}^{\epsilon}(p^{2},\psi\otimes\chi\otimes \chi_{p})$.
\end{lem}

\begin{proof}
	We can write
	\[
	G_{\psi} = \frac{1}{2p}\sum_{\gamma (p)^{*}}\psi(\gamma)\sum_{j(p)}G\bigg|_{k}\begin{pmatrix}1 &j/p \\ 0 &1 \end{pmatrix}\mathbf{e}(jQ(\gamma)).
	\]
	Let $M = \left(\begin{smallmatrix}a & b \\ c & d \end{smallmatrix}\right) \in \Gamma_{0}(p^{2})$. We compute
	\[
	\begin{pmatrix}1 &j/p \\ 0 &1 \end{pmatrix}\begin{pmatrix}a & b \\ c & d \end{pmatrix}\begin{pmatrix}1 &-d^{2}j/p \\ 0 &1 \end{pmatrix} = \begin{pmatrix}a + cj/p & b+(1-ad)dj/p-d^{2}cj^{2}/p^{2} \\ c & d-d^{2}cj/p \end{pmatrix} \in \Gamma_{0}(p^{2}).
	\]
	Note that the $d$-entry of this matrix equals $d$ mod $p$. Hence we obtain
	\begin{align*}
	G_{\psi}|_{k}M = \chi\chi_{p}(d)\frac{1}{2p}\sum_{\gamma \in (p)^{*}}\psi(\gamma)\sum_{j(p)}G\bigg|_{k}\begin{pmatrix}1 &d^{2}j/p \\ 0 &1 \end{pmatrix}\mathbf{e}(jQ(\gamma)).
	\end{align*}
	Replacing $d^{2}j$ by $j$ and then $\gamma$ by $d\gamma$ gives a factor $\psi(d)$ and completes the proof.
\end{proof}

We let $W_{p^{2}} = \left(\begin{smallmatrix}0 & - 1 \\ p^{2} & 0 \end{smallmatrix} \right)$ be the  Atkin-Lehner (or Fricke) involution. It maps $A_{k}(p^{2},\chi\otimes\chi_{p})$ to $A_{k}(p^{2},\overline{\chi}\otimes \chi_{p})$, but it does in general not respect the $\epsilon$-condition. We say that $G \in A_{k}^{\epsilon}(p^{2},\chi\otimes\chi_{p})$ satisfies the \textbf{Atkin-Lehner condition} if the twists of $G$ by all Dirichlet characters $\psi$ mod $p$ with $\psi \neq \overline{\chi}$ satisfy
\begin{align}\label{AL}
G_{\psi}|_{k}W_{p^{2}} = \overline{\psi\chi(2\alpha)}\frac{g(\psi\chi)}{\sqrt{p}}\mathbf{e}(\sig(A,Q)/8)G_{\overline{\psi}\overline{\chi}^{2}},
\end{align}
where $g(\psi\chi) = \sum_{n(p)^{*}}\psi\chi(n)\mathbf{e}(n/p)$ is a Gauss sum. We let $A_{k}^{\epsilon,\AL}(p^{2},\chi\otimes\chi_{p})$ be the subspace of $A_{k}^{\epsilon}(p^{2},\chi\otimes \chi_{p})$ consisting of all forms satisfying the Atkin-Lehner condition.

\begin{prop}\label{prop isomorphism p}
	The linear map
	\[
	\varphi_{\chi}: A_{k}(\rho^{*}) \to A_{k}^{\epsilon,\AL}(p^{2},\chi\otimes \chi_{p}), \qquad F(\tau) = \sum_{\gamma(p)}F_{\gamma}(\tau)\e_{\gamma} \mapsto \sum_{\gamma(p)^{*}}\chi(\gamma)F_{\gamma}(p\tau),
	\]
	is an isomorphism. The inverse map is given by
	\[
	\varphi_{\chi}^{-1}: G(\tau) = \sum_{\gamma(p)^{*}}G_{\gamma}(\tau) \mapsto \sum_{\gamma(p)^{*}}\overline{\chi(\gamma)}G_{\gamma}(\tau/p)\e_{\gamma} + G_{0}(\tau/p)\e_{0},
	\]
	where $G_{0}$ is defined by
	\[
	G_{0}(\tau) = \frac{\sqrt{p}}{p-1}\mathbf{e}(-\sig(A,Q)/8)\left(G_{\overline{\chi}}|_{k}W_{p^{2}}\right)(\tau)+\frac{1}{p-1}G_{\overline{\chi}}(\tau).
	\]
\end{prop}

\begin{proof}
	We first show that $G = \varphi_{\chi}(F)$ for $F = \sum_{\gamma(p)}F_{\gamma}\e_{\gamma} \in A_{k}(\rho^{*})$ satisfies the Atkin-Lehner condition. Let $\psi$ be a Dirichlet character mod $p$ and let $\delta_{\psi,\overline{\chi}} = 1$ if $\psi = \overline{\chi}$ and $\delta_{\psi,\overline{\chi}} = 0$ otherwise. We compute
	\begin{align*}
	\left(G_{\psi}|_{k}W_{p^{2}}\right)(\tau) &= p^{k}(p^{2}\tau )^{-k}G_{\psi}\left(-\frac{1}{p^{2}\tau}\right) \\
	&= \sum_{\gamma(p)^{*}}\psi(\gamma)\chi(\gamma)(p\tau)^{-k}F_{\gamma}\left(-\frac{1}{p\tau} \right) \\
	&= \sum_{\gamma(p)^{*}}\psi(\gamma)\chi(\gamma)\frac{\mathbf{e}(\sig(A,Q)/8)}{\sqrt{p}}\sum_{\beta(p)}\mathbf{e}((\beta,\gamma))F_{\beta}(p\tau) \\
	&= \frac{\mathbf{e}(\sig(A,Q)/8)}{\sqrt{p}}\left(\sum_{\beta(p)^{*}}\left(\sum_{\gamma(p)^{*}}\psi(\gamma)\chi(\gamma)\mathbf{e}(2\alpha \beta \gamma/p) \right) F_{\beta}(p\tau)+ \delta_{\psi,\overline{\chi}}(p-1)F_{0}(p\tau)\right)\\
	&= \frac{\mathbf{e}(\sig(A,Q)/8)}{\sqrt{p}}\left(\sum_{\gamma (p)^{*}}\psi(\gamma)\chi(\gamma)\mathbf{e}(2\alpha \gamma/p)\sum_{\beta(p)^{*}}\overline{\psi(\beta)}\overline{\chi(\beta)} F_{\beta}(p\tau) + \delta_{\psi,\overline{\chi}}(p-1)F_{0}(p\tau)\right)\\
	&= \frac{\mathbf{e}(\sig(A,Q)/8)}{\sqrt{p}}\left(\overline{\psi\chi(2\alpha)}g(\psi\chi)G_{\overline{\psi}\overline{\chi}^{2}}(\tau) + \delta_{\psi,\overline{\chi}}(p-1)F_{0}(p\tau)\right).
	\end{align*}
	This shows that $G$ satisfies the Atkin-Lehner condition, i.e., $G \in A_{k}^{\epsilon,\AL}(p^{2},\chi\otimes \chi_{p})$.
	
	Conversely, if $G \in A_{k}^{\epsilon,\AL}(p^{2},\chi\otimes\chi_{p})$ and if $G_{0}$ is defined as in the proposition, then we can reverse the above computation (with $F_{0}(p\tau) = G_{0}(\tau)$ and $F_{\gamma}(p\tau) = \overline{\chi(\gamma)}G_{\gamma}(\tau)$ for $\gamma \neq 0$) to see that the second and third line agree for all Dirichlet characters $\psi$ mod $p$. By character orthogonality, we obtain that 
	\[
	F_{\gamma}|_{k}S = \frac{\mathbf{e}(\sig(A,Q)/8)}{\sqrt{p}}\sum_{\beta(p)}\mathbf{e}((\beta,\gamma))F_{\beta}
	\]
	for all $\gamma \neq 0$. A short computation shows that this equation also implies
	\[
	F_{0}|_{k}S = \frac{\mathbf{e}(\sig(A,Q)/8)}{\sqrt{p}}\sum_{\beta(p)}F_{\beta}.
	\]
	Furthermore, it is easy to check that $F_{\gamma}$ transforms correctly under $T$. We find that $\varphi_{\chi}^{-1}(G) \in A_{k}(\rho^{*})$. Since $\varphi_{\chi}\circ \varphi_{\chi}^{-1} = \mathrm{id}$ and $\varphi_{\chi}$ is injective, $\varphi_{\chi}$ is an isomorphism.
\end{proof}

\subsection{Finite quadratic modules of order $2p$} Suppose that $|A| = 2p$ with an odd prime $p$. Then $A \cong \Z/2p\Z$ with the quadratic form
\[
Q(p\gamma_{1} + \gamma_{2}) = \delta \gamma_{1}^{2}/4+\alpha \gamma_{2}^{2}/p
\]
for $\gamma_{1} \in \Z/2\Z$ and $\gamma_{2} \in \Z/p\Z$, where $\delta \in \{\pm 1\}$ and $\alpha \in \Z$ with $p \nmid \alpha$. Set $\epsilon = \chi_{p}(\alpha)$. Using the quadratic Gauss sum and Milgram's formula we obtain $(1+\delta i)\epsilon\varepsilon_{p} = \sqrt{2}\mathbf{e}(\sig(A,Q)/8)$, so the signature $\sig(A,Q) \in\Z/8\Z$ is given in terms of $p, \epsilon$ and $\delta$ as follows:
\begin{align*}
	\begin{array}{c||c|c}
	p \pmod 4 & 1 & 3 \\
	\hline\hline
	\epsilon = +1 & \delta & \delta+2\\
	\hline
	\epsilon = -1 & \delta+4 & \delta+6
	\end{array}
\end{align*}
Now $\sig(A,Q)$ is odd. Hence we can assume that $k$ is half-integral since otherwise $A_{k}(\rho^{*}) = 0$.

Let us briefly recall the definition of modular forms of half-integral weight. The theta multiplier is given by
\[
\nu_{\vartheta}(M)= \bigg(\frac{c}{d}\bigg)\varepsilon_{d}^{-1}, \qquad \varepsilon_{d} = \bigg(\frac{2}{d}\bigg)\mathbf{e}((1-d)/8)= \begin{cases} 1, & d \equiv 1(4), \\
i,&  d \equiv 3 (4),  \end{cases}
\]
for $M = \left(\begin{smallmatrix}a & b \\ c & d \end{smallmatrix}\right) \in \Gamma_{0}(4)$. A function $G:\H \to \C$ is called a weakly holomorphic modular form of weight $k \in \frac{1}{2}+\Z$ for $\Gamma_{0}(N)$ with $4 \mid N$ and character $\chi$ mod $N$ if it is holomorphic on $\H$ and meromorphic at the cusps, and if it transforms as
\[
G(M\tau) = \chi(M)\nu_{\vartheta}(M)^{\pm 1}(c\tau+d)^{k}G(\tau)
\] 
for $M = \left(\begin{smallmatrix}a & b \\ c & d \end{smallmatrix}\right) \in \Gamma_{0}(N)$, where the sign in $\nu_{\vartheta}(M)^{\pm 1}$ is chosen such that $\nu_{\vartheta}(-1)^{\pm 1}\chi(-1)i^{2k} = 1$. We denote the space of all these functions by $A_{k}(N,\chi)$. 

We let $\chi$ be a character mod $p$ and we again assume that 
\[
\chi(-1) = (-1)^{k+\sig(A,Q)/2}.
\]
We consider the space
\begin{align*}
A_{k}^{\epsilon}(16p^{2},\chi)= \left\{\sum_{n\gg -\infty}c(n)q^{n} \in A_{k}(16p^{2},\chi) \ : \  \begin{aligned} & c(n) = 0 \text{ unless } n \in 4p(\Z-Q(\gamma)) \\ & \text{for some } \gamma \in (\Z/2p\Z) \setminus \{0,p\} \end{aligned}\right\}
\end{align*}
and we let $A_{k}^{\epsilon}(4p^{2},\chi) = A_{k}^{\epsilon}(16p^{2},\chi) \cap A_{k}(4p^{2},\chi)$.

We define the \textbf{twisted component sum} of $F(\tau) = \sum_{\gamma(2p)}F_{\gamma}(\tau)\e_{\gamma} \in A_{k}(\rho^{*})$ by
\[
\varphi_{\chi}\left(\sum_{\gamma(2p)}F_{\gamma}(\tau)\e_{\gamma} \right) = \sum_{\gamma(2p)}\chi(\gamma)F_{\gamma}(4p\tau).
\]
Note that, since $\chi$ is a character mod $p$, we discard the components $F_{0}$ and $F_{p}$ in the twisted component sum. If $k$ is an antisymmetric weight, then $F_{0} = F_{p} = 0$ is automatic. This map was already suggested in \cite{EZ}, p. 70, in the context of Jacobi forms.

\begin{prop}
	If $F \in A_{k}(\rho^{*})$, then $\varphi_{\chi}(F) \in A_{k}^{\epsilon}(4p^{2}, \chi)$. Furthermore, $\varphi_{\chi}$ is injective.
\end{prop}

\begin{proof}
	We first show that that $\varphi_{\chi}(F)$ transforms correctly under $\Gamma_{0}(16p^{2})$. For $M = \left( \begin{smallmatrix}a & b \\ c & d  \end{smallmatrix}\right) \in \Gamma_{0}(16p^{2})$ we let $N = \left( \begin{smallmatrix}a & 4pb \\ c/4p & d  \end{smallmatrix}\right)$. Then we compute
	\[
	\varphi_{\chi}(F)(M\tau) = \sum_{\gamma(2p)}\chi(\gamma)F_{\gamma}(4p\cdot M\tau) = \sum_{\gamma(2p)}\chi(\gamma)F_{\gamma}(N(4p\tau)).
	\]
	Using \cite{BReflection}, Theorem~5.2, we obtain
	\begin{align*}
	F_{\gamma}(N(4p\tau)) &= \bigg(\frac{c/4p}{d}\bigg)\bigg( \frac{d}{2p}\bigg)\mathbf{e}((1-d)\delta/8)(c\tau+d)^{k}F_{d^{-1}\gamma}(4p\tau).
	\end{align*}
	We compute
	\begin{align*}
	\bigg(\frac{c/4p}{d}\bigg)\bigg( \frac{d}{2p}\bigg)\mathbf{e}((1-d)\delta/8) &=\bigg(\frac{p}{d}\bigg)\bigg(\frac{d}{p}\bigg)\bigg( \frac{c}{d}\bigg)\bigg( \frac{2}{d}\bigg)\mathbf{e}((1-d)\delta/8) = \bigg(\frac{p}{d}\bigg)\bigg( \frac{d}{p}\bigg)\nu_{\vartheta}(M)^{-\delta}.
	\end{align*}
	By quadratic reciprocity we have $\big(\frac{p}{d}\big)\big( \frac{d}{p}\big) = 1$ if $p \equiv 1 \pmod 4$ and
	\[
	\bigg(\frac{p}{d}\bigg)\bigg( \frac{d}{p}\bigg) = \begin{cases}1, & \text{if } d \equiv 1 \pmod 4, \\
	-1, & \text{if } d\equiv 3 \pmod 4, \end{cases}
	\]
	if $p \equiv 3 \pmod 4$. This gives the stated transformation behaviour under $\Gamma_{0}(16p^{2})$.

	In order to show the transformation behaviour under $\Gamma_{0}(4p^{2})$, it suffices to check the transformation under the matrices $U_{4p^{2}j} = \left(\begin{smallmatrix}1 & 0 \\ 4p^{2}j &1 \end{smallmatrix}\right)$ for $j = 0,1,2,3$ since they represent $\Gamma_{0}(16p^{2})\backslash \Gamma_{0}(4p^{2})$. We compute
	\begin{align*}
	\varphi_{\chi}(F)(U_{4p^{2}j}\tau) &= \sum_{\gamma(2p)}\chi(\gamma)F_{\gamma}(4p\cdot U_{4p^{2}j} \tau) \\
	&=  \sum_{\gamma(2p)}\chi(\gamma)F_{\gamma}(U_{pj}(4p\tau)) \\
	&= \sum_{\gamma(2p)}\chi(\gamma)F_{\gamma}(S^{-1}T^{-pj}S(4p\tau)) \\
	&= (4p^{2}j\tau+1)^{k}\frac{1}{2p}\sum_{\gamma(2p)}\chi(\gamma)\sum_{\beta(2p)}\mathbf{e}(-(\beta,\gamma))\mathbf{e}(jpQ(\beta))\sum_{\mu(2p)}\mathbf{e}((\mu,\beta))F_{\mu}(4p\tau).
	\end{align*} 
	If we write $\gamma = p\gamma_{1}+\gamma_{2}$ with $\gamma_{1} \in \Z/2\Z$ and $\gamma_{2} \in \Z/p\Z$, and similarly for $\beta$, and use that $\chi$ only depends on $\gamma_{2}$, we see that the sum over $\gamma_{1}$ vanishes unless $\beta_{2} = 0$. This means that we can replace $\mathbf{e}(jpQ(\beta))$ by $1$ in the above sum. But then the sum over $\beta$ equals $2p$ if $\mu = \gamma$, and vanishes otherwise. Hence we get
	\begin{align*}
	\varphi_{\chi}(F)(U_{4p^{2}j}\tau)&=(4p^{2}j\tau+1)^{k}\sum_{\gamma(2p)}\chi(\gamma)F_{\gamma}(4p\tau) =(4p^{2}j\tau+1)^{k}\varphi_{\chi}(F)(\tau).
	\end{align*}
	This shows that $\varphi_{\chi}(F)$ transforms correctly under $\Gamma_{0}(4p^{2})$. It is easy to see that $\varphi_{\chi}$ is holomorphic on $\H$ and meromorphic at the cusps, and that it satisfies the $\epsilon$-condition.
	
	Now suppose that $\varphi_{\chi}(F) = 0$ for some $F \in A_{k}(\rho^{*})$. By comparing the index sets on which the components $F_{\gamma}$ are supported, we obtain that $F_{\gamma} = 0$ for $\gamma \notin \{0,p\}$. Then the transformation behaviour of $F$ under $S$ implies 
	\[
	F_{0}(-1/\tau) = \tau^{k}\frac{\mathbf{e}(\sig(A,Q)/8)}{\sqrt{2p}}(F_{0}(\tau)+F_{p}(\tau)), \qquad 	F_{p}(-1/\tau) = \tau^{k}\frac{\mathbf{e}(\sig(A,Q)/8)}{\sqrt{2p}}(F_{0}(\tau)-F_{p}(\tau)).
	\]
	Applying $\tau \mapsto -1/\tau$ a second time, we get $i^{2k}F_{0}= \frac{\mathbf{e}(\sig(A,Q)/4)}{p}F_{0}$ and $i^{2k}F_{p} = \frac{\mathbf{e}(\sig(A,Q)/4)}{p}F_{p}$, hence $F_{0} = F_{p} = 0$. Thus $F = 0$ and $\varphi_{\chi}$ is injective.
\end{proof}

We split $G = \sum_{n\gg-\infty}c(n)q^{n}\in A_{k}^{\epsilon}(16p^{2},\chi)$ into components
\[
G(\tau) = \sum_{\gamma(2p)}G_{\gamma}(\tau), \qquad G_{\gamma}(\tau) = \frac{1}{2}\sum_{\substack{n \in 4p(\Z - Q(\gamma)) \\ n \gg -\infty}}c(n)q^{n},
\]
and define its \textbf{component-wise twist} by a Dirichlet character $\psi$ mod $p$ by 
\[
G_{\psi}(\tau) = \sum_{\gamma(2p)}\psi(\gamma)G_{\gamma}(\tau).
\]

\begin{lem}
	Let $G \in A_{k}^{\epsilon}(16p^{2},\chi)$ and let $\psi$ be a Dirichlet character mod $p$. Then $G_{\psi} \in A_{k}^{\epsilon}(16p^{2},\psi\otimes \chi)$.
\end{lem}

\begin{proof}
	The proof is analogous to the proof of Lemma~\ref{lem twisting}, so we leave the details to the reader.
\end{proof}

In contrast to the case $|A| = p$ we need another notion to describe the image of $\varphi_{\chi}$. We call $\gamma \in \Z/2p\Z$ even if $4pQ(\gamma)$ is even, and odd if $4pQ(\gamma)$ is odd. The \textbf{even} and \textbf{odd parts} of $G \in A_{k}^{\epsilon}(16p^{2},\chi)$ are defined by
\begin{align*}
	G^{\even}(\tau) = \sum_{\substack{\gamma (2p) \\ \gamma \, \even}}G_{\gamma}(\tau), \qquad G^{\odd}(\tau) = \sum_{\substack{\gamma (2p) \\ \gamma \, \odd}}G_{\gamma}(\tau).
\end{align*}
Note that taking the even and odd parts of $G$ commutes with component-wise twisting.

\begin{lem}
	If $G \in A_{k}^{\epsilon}(16p^{2},\chi)$ then $G^{\even},G^{\odd} \in A_{k}^{\epsilon}(16p^{2},\chi)$ as well.
\end{lem}

\begin{proof}
	We can write
	\[
	G^{\even}(\tau) = \frac{1}{2}\left(G(\tau)+ G(\tau+1/2) \right), \qquad G^{\odd}(\tau) = \frac{1}{2}\left(G(\tau)- G(\tau+1/2) \right).
	\]
	For $\left(\begin{smallmatrix}a & b \\ c & d\end{smallmatrix}\right) \in \Gamma_{0}(16p^{2})$ we have
	\[
	\begin{pmatrix}1 & 1/2 \\ 0 & 1\end{pmatrix}\begin{pmatrix}a & b \\ c & d\end{pmatrix}\begin{pmatrix}1 & -1/2 \\ 0 & 1\end{pmatrix} = \begin{pmatrix}a+c/2 & b+(d-a)/2-c/4 \\ c & d-c/2\end{pmatrix} \in \Gamma_{0}(16p^{2}),
	\]
	which easily implies $G(\tau+1/2) \in A_{k}^{\epsilon}(16p^{2},\chi)$. This proves the lemma.
\end{proof}

For $G \in A_{k}(16p^{2},\chi)$ we define the Atkin-Lehner involution 
\[
G|_{k}W_{16p^{2}} = (4p)^{-k}(-i\tau)^{-k}G(-1/16p^{2}\tau).
\]
Then $G|_{k}W_{16p^{2}} \in A_{k}(16p^{2},\overline{\chi})$ and $G|_{k}W_{16p^{2}}|_{k}W_{16p^{2}} = G$. In general, the Atkin-Lehner involution does not preserve the $\epsilon$-condition. We say that $G \in A_{k}^{\epsilon}(16p^{2},\chi)$ satisfies the \textbf{Atkin-Lehner condition} if its twists by all Dirichlet characters $\psi$ mod $p$ with $\psi \neq \overline{\chi}$ satisfy
\begin{align*}
	G_{\psi}|_{k}W_{16p^{2}} &= \overline{\psi\chi(2\alpha)}\frac{\sqrt{2}g(\psi\chi)}{\sqrt{p}}\mathbf{e}(\sig(A,Q)/8)i^{k}G_{\overline{\psi}\overline{\chi}^{2}}^{\even},
\end{align*}
where $g(\psi\chi) = \sum_{n (p)^{*}}\psi\chi(n)\mathbf{e}(n/p)$ is a Gauss sum. Let $A_{k}^{\epsilon,\AL}(16p^{2},\chi)$ be the subspace of $A_{k}^{\epsilon}(16p^{2},\chi)$ satisfying the Atkin-Lehner condition. Note that, after applying $W_{16p^{2}}$ and a short calculation, the Atkin-Lehner condition also implies that
\begin{align*}
(G_{\psi}^{\even}-G_{\psi}^{\odd})|_{k}W_{16p^{2}} &= \overline{\psi\chi(2\alpha)}\frac{\sqrt{2}g(\psi\chi)}{\sqrt{p}}\mathbf{e}(\sig(A,Q)/8)i^{k}G_{\overline{\psi}\overline{\chi}^{2}}^{\odd}
\end{align*}
for $\psi \neq \overline{\chi}$.

\begin{prop}\label{prop isomorphism 2p}
	The linear map
	\[
	\varphi_{\chi}: A_{k}(\rho^{*}) \to A_{k}^{\epsilon,\AL}(16p^{2},\chi), \qquad F(\tau) = \sum_{\gamma(2p)}F_{\gamma}(\tau)\e_{\gamma} \mapsto \sum_{\gamma(2p)}\chi(\gamma)F_{\gamma}(4p\tau),
	\]
	is an isomorphism. The inverse map is given by
	\[
	\varphi_{\chi}^{-1}: G(\tau) = \sum_{\gamma(2p)}G_{\gamma}(\tau) \mapsto \sum_{\substack{\gamma(2p) \\ \gamma \neq 0,p (2p)}}\overline{\chi(\gamma)}G_{\gamma}(\tau/4p)\e_{\gamma} + G_{0}(\tau/4p)\e_{0} + G_{p}(\tau/4p)\e_{p},
	\]
	where $G_{0}$ and $G_{p}$ are defined by
	\begin{align*}
	G_{0}(\tau) &= \frac{\sqrt{p}}{\sqrt{2}(p-1)}\mathbf{e}(-\sig(A,Q)/8)i^{-k}(G_{\overline{\chi}}^{\even}(\tau)+G_{\overline{\chi}}^{\odd}(\tau))|_{k}W_{16p^{2}} +\frac{1}{p-1} G_{\overline{\chi}}^{\even}(\tau),\\
	G_{p}(\tau) &= \frac{\sqrt{p}}{\sqrt{2}(p-1)}\mathbf{e}(-\sig(A,Q)/8)i^{-k}(G_{\overline{\chi}}^{\even}(\tau)-G_{\overline{\chi}}^{\odd}(\tau))|_{k}W_{16p^{2}} +\frac{1}{p-1} G_{\overline{\chi}}^{\odd}(\tau).
	\end{align*}
\end{prop}

\begin{proof}
	The proof is very similar to the proof of Proposition~\ref{prop isomorphism p}, so we omit it for brevity.
\end{proof}

\section{Application: the Doi-Naganuma lift and theta contraction} \label{sec Doi-Naganuma}

Let $K = \mathbb{Q}(\sqrt{p})$ for a prime $p \equiv 1 \pmod 4$ and let $\mathcal{O}_K$ be its ring of integers. Let $\mathcal{O}_K^{\#} = (1 / \sqrt{p}) \mathcal{O}_K$ be the dual lattice of $\mathcal{O}_K$ with respect to the trace and consider the finite quadratic module $(\mathcal{O}_K^{\#}/\mathcal{O}_K,-N_{K/\mathbb{Q}})$. It has order $p$ and signature $0$ mod $8$. Let $\lambda \in \mathcal{O}_K$ be any totally positive prime with norm $\ell = N_{K/\mathbb{Q}}(\lambda)$, and let $b \in \mathbb{Z}$ be any integer with $b^2 \equiv p \pmod {4\ell}$ (which exists by quadratic reciprocity). Since $\lambda$ is prime, one of $(b + \sqrt{p}) / 2\lambda, (b - \sqrt{p})/2\lambda$ is integral; it is then straightforward to show that $\{\lambda', (b \pm \sqrt{p})/2\lambda\}$ is a $\mathbb{Z}$-basis of $\mathcal{O}_K$ and 
\[
	\mathbf{S} = \begin{pmatrix} -2\ell & -b \\ -b & \frac{p - b^2}{2\ell} \end{pmatrix}
\]
is the Gram matrix of $-N_{K/\mathbb{Q}}$ in that basis. 

If $\ell \neq p$, then for each $r \in \Z/p\Z$ there exists a unique $a \in \Z/2\ell\Z$ with $r \equiv ab \pmod {2\ell}$. We fix a bijection $\mathcal{O}_{K}^{\#}/\mathcal{O}_{K}\cong \Z/p\Z$ by sending $r \in \Z/p\Z$ to the element 
\[
\gamma_{a,r}+\mathcal{O}_{K} = \frac{a \pm r/\sqrt{p}}{2 \lambda}+\mathcal{O}_{K} \in \mathcal{O}_{K}^{\#}/\mathcal{O}_{K}.
\]
If $\ell = p$, then we fix the bijection which identifies $a \in \mathbb{Z}/p\mathbb{Z}$ with $\gamma_a + \mathcal{O}_K = \frac{a(1 + \sqrt{p})}{2 \lambda}+\mathcal{O}_{K}$ instead. (In this case, one can always take $b = p$ and $\lambda = \varepsilon \sqrt{p}$ if $\varepsilon$ is the fundamental unit of $\mathcal{O}_K$.) \\

The \textbf{theta decomposition} identifies vector-valued modular forms for the dual Weil representation attached to $(\mathcal{O}_K^{\#}/\mathcal{O}_K,-N_{K/\mathbb{Q}})$ with vector-valued Jacobi forms of fractional index $p/4\ell$ for the dual Weil representation attached to the discriminant form with Gram matrix $(-2\ell)$ and a particular representation of the Heisenberg group, see e.g. \cite{W}. By setting the Heisenberg variable of those Jacobi forms equal to zero one obtains the \textbf{theta contraction}, a graded homomorphism between the modular forms $M_*(-N_{K/\mathbb{Q}})$ and $M_{*+1/2}((-2\ell))$ as graded modules over the ring $M_*(\mathrm{SL}_2(\mathbb{Z}))$ of scalar-valued modular forms. This was introduced by Ma \cite{M} in order to study the quasi-pullback of Borcherds products. Explicitly in terms of Fourier coefficients, it is the map 
\[
	\Theta : M_k(-N_{K/\mathbb{Q}}) \to M_{k+1/2}((-2\ell)), \quad \sum_{\gamma \in \mathcal{O}_K^{\#}/\mathcal{O}_K} \sum_{n \in \mathbb{Z} + N_{K/\mathbb{Q}}(\gamma)} c(n,\gamma) q^n \mathfrak{e}_{\gamma} \mapsto \sum_{a \in \mathbb{Z}/2\ell \mathbb{Z}} \sum_{n \in \mathbb{Z} + a^2 / 4 \ell} \tilde c(n,a) q^n \mathfrak{e}_a,
\]
where 
\[
	\tilde c(n,a) = \sum_{r \equiv ab \, (2 \ell)} c\left( n - \frac{r^2}{4 \ell p}, \gamma_{a,r}\right).
\]

Possibly the most important aspect of the theta contraction (which can be defined more generally) is that it fits into a commutative diagram involving the additive theta lift (of Oda and Rallis-Schiffmann) and restriction to Heegner divisors: letting $\Lambda$ be an even lattice of type $(2,b^-)$ such that $b^{-} \geq 2$ is greater than the Witt rank of $\Lambda$, and $\lambda^{\perp}$ the orthogonal complement of a primitive, negative-norm vector $\lambda \in \Lambda$, the natural pullback map $\mathrm{Res}$ for orthogonal modular forms satisfies
\begin{center} \begin{tikzpicture}[node distance=2.5cm, auto]
\node (A) {$S_{k + 1 - b^-/2}(\rho^*_{\Lambda})$};
\node (AA) [right of=A] {$ $};
\node (B) [right of=AA] {$S_{k}(O(\Lambda))$};
\node (C) [below of=A] {$S_{k+1 - (b^- - 1)/2}(\rho^*_{\lambda^{\perp}})$};
\node (D) [below of=B] {$S_{k}(O(\lambda^{\perp}))$};
\draw[->] (B) to node {$\text{Res}$} (D);
\draw[->] (A) to node {$\Theta$} (C);
\draw[->] (A) to node {$\text{Theta lift}$} (B);
\draw[->] (C) to node {$\text{Theta lift}$} (D);
\end{tikzpicture}
\end{center}
for $k \geq 2$.

For the Doi-Naganuma lift (i.e. $b^{-} =2$) this can be made very explicit. Recall that for a cusp form 
\[
	F(\tau) = \sum_{\gamma \in \mathcal{O}_K^{\#} / \mathcal{O}_K} \sum_{n \in \mathbb{Z} + N_{K/\mathbb{Q}}(\gamma)} c(n,\gamma) q^n \mathfrak{e}_{\gamma} \in S_k(-N_{K/\mathbb{Q}})
\]
of weight $k \ge 2$, the Doi-Naganuma lift is a Hilbert cusp form $\Phi_F$ of weight $k$ for $\mathrm{SL}_2(\mathcal{O}_K)$ with Fourier expansion 
\[
	\Phi_F(\tau_1,\tau_2) = \sum_{\substack{\nu \in \mathcal{O}_K^{\#} \\ \nu,\nu' > 0}}\sum_{n=1}^{\infty} c(\nu \nu', \nu) n^{k-1} \mathbf{q}^{n \nu}, \quad \text{where} \; \; \mathbf{q}^{\nu} = e^{2\pi i (\nu \tau_1 + \nu' \tau_2)}.
\]
Moreover $\Phi_F$ satisfies the graded symmetry $\Phi_F(\tau_1,\tau_2) = (-1)^k \Phi_F(\tau_2,\tau_1)$. 

With $\lambda$ as above, there is a natural restriction map onto a component of the Hirzebruch-Zagier curve $T_{\ell}$: 
\[
	\mathrm{Res} : S_k(\mathrm{SL}_2(\mathcal{O}_K)) \to S_{2k}(\Gamma_0(\ell)), \quad f(\tau_1,\tau_2) \mapsto f(\lambda \tau, \lambda' \tau).
\]
It turns out that $\mathrm{Res}(\Phi_{F})$ equals the Shimura lift
\[
\sum_{a=1}^{\infty}\sum_{n=1}^{\infty} \tilde c(a^2 / 4\ell, a) n^{k-1} q^{na}
\]
of the contracted form $\Theta F = \sum_{a,n}\tilde c(n,a)q^{n}\e_{a} \in S_{k+1/2}((-2\ell))$.

\begin{lem}\label{lem doi-naganuma-shimura}
	We have the following commutative diagram:
\begin{center} \begin{tikzpicture}[node distance=2.5cm, auto]
\node (A) {$S_{k}(-N_{K/\Q})$};
\node (AA) [right of=A] {$ $};
\node (B) [right of=AA] {$S_k(\SL_{2}(\mathcal{O}_{K}))$};
\node (C) [below of=A] {$S_{k+1/2}((-2\ell))$};
\node (D) [below of=B] {$S_{2k}(\Gamma_{0}(\ell))$};
\draw[->] (B) to node {$\mathrm{Res}$} (D);
\draw[->] (A) to node {$\Theta$} (C);
\draw[->] (A) to node {$\mathrm{Doi-Naganuma}$} (B);
\draw[->] (C) to node {$\mathrm{Shimura}$} (D);
\end{tikzpicture}
\end{center}
\end{lem}

\begin{proof} Since the elements $\nu \in \mathcal{O}_K^{\#}$ with $\mathrm{Tr}(\nu \lambda) = a \in \mathbb{N}$ are exactly those of the form $\gamma_{a,r} = \frac{a \pm r / \sqrt{p}}{2 \lambda}$ with $r \equiv ab$ mod $2\ell$, we find 
\begin{align*} 
	\Phi_F(\lambda \tau, \lambda' \tau) &= \sum_{n=1}^{\infty} \sum_{\mathrm{Tr}(\nu \lambda) = a} c(\nu \nu', \nu) n^{k-1} q^{na} \\ &= \sum_{n=1}^{\infty} \sum_{a=1}^{\infty} \sum_{r \equiv ab \, (2\ell)} c \left( \frac{a^2}{4\ell} - \frac{r^2}{4 \ell p}, \frac{a \pm r / \sqrt{p}}{2 \lambda} \right) n^{k-1} q^{na} \\ &= \sum_{n,a=1}^{\infty} \tilde c(a^2 / 4\ell, a) n^{k-1} q^{na}, 
\end{align*}
i.e. $\mathrm{Res}(\Phi_{F})(\tau) = \Phi_F(\lambda \tau, \lambda' \tau)$ is the Shimura lift of the contracted form $\Theta F \in S_{k+1/2}((-2\ell)).$
\end{proof}

In this section we observe that this relationship takes a simple form in terms of twisted component sums of $F$ and $\Theta F$. Recall that for a $q$-series $f(\tau) = \sum_n c(n) q^n$ the Hecke operator $U_p$ is defined by 
\[
	f|U_p(\tau) = \sum_n c(pn) q^n = \sum_{n \equiv 0 \, (p)} c(n) q^{n/p}.
\]

\begin{prop}\label{prop theta contraction}~
	\begin{enumerate}
		\item[(i)] Suppose $\ell \ne p$. Let  $\psi_{\ell}$ and $\psi_p$ be Dirichlet characters modulo $\ell$ and $p$ with $\psi_{\ell}(-1) = \psi_{p}(-1) = (-1)^{k}$ and let $\chi$ be the Dirichlet character modulo $\ell p$ defined by $\chi(r) = \psi_{\ell}(r) \psi_p(r)$ for all $r \in \mathbb{Z}$. By abuse of notation, let $\psi_p$ denote the ``character" on the cosets $\mathcal{O}_K^{\#} / \mathcal{O}_K$ defined by setting $\psi_p((a \pm r / \sqrt{p}) / 2 \lambda) = \psi_p(r)$ for any $a,r \in \mathbb{Z}$ satisfying $r \equiv ab$ mod $2 \ell$. Then 
		\[
			\varphi_{\psi_{\ell}} (\Theta F)(\tau) = \frac{1}{2\psi_{\ell}(b)} \cdot \Big( \varphi_{\overline{\psi_p}} (F)(4\ell \tau) \cdot \vartheta_{\chi}(\tau) \Big) \Big| U_p
		\]
		where $\vartheta_{\chi}(\tau) = \sum_{r \in \mathbb{Z}} \chi(r) q^{r^2}$ is the twisted Jacobi theta series.
	\item[(ii)] Suppose $\ell = p$, and let $\psi_p$ be a Dirichlet character mod $p$ with $\psi_{p}(-1) = (-1)^{k}$. By abuse of notation, define $\psi_p$ on $\mathcal{O}_K^{\#}/\mathcal{O}_K$ by $\psi_p(a / \lambda) = \psi_p(2a)$, $a \in \mathbb{Z}/p\mathbb{Z}$, where $\lambda = \varepsilon \sqrt{p}$ and $\varepsilon$ is the fundamental unit of $\mathcal{O}_K$. Then $$\varphi_{\psi_p} (\Theta F)(\tau) = \Big( \varphi_{\psi_p} (F)(4p\tau) \cdot \vartheta(\tau) \Big) \Big| U_p,$$ where $\vartheta(\tau) = \sum_{r \in \mathbb{Z}} q^{r^2}.$
	\end{enumerate}
\end{prop}

\begin{proof} ~
	\begin{enumerate}
		\item[(i)] Write $f(\tau) = \varphi_{\overline{\psi_p}} (F)(\tau) = \sum_{\gamma \in \mathcal{O}_K^{\#} / \mathcal{O}_K} \sum_{n \in \mathbb{Z} + N_{K/\mathbb{Q}}(\gamma)} \overline{\psi_p(\gamma)} c(n,\gamma) q^{pn}$. In the product 
		\begin{align*} 
			f(4 \ell \tau) \vartheta_{\chi}(\tau) &= \Big( \sum_{\gamma,n} \overline{\psi_p(\gamma)} c(n,\gamma) q^{4\ell pn} \Big) \Big( \sum_{r=-\infty}^{\infty} \chi(r) q^{r^2} \Big) \\ &= \sum_{\gamma} \sum_{n,r} \overline{\psi_p(\gamma)} \chi(r) c(n, \gamma) q^{4 \ell pn + r^2}, 
			\end{align*} 
			we get exponents which are divisible by $p$ only when $n \in \mathbb{Z} + a^2 / 4\ell$ and $r \equiv \pm ab \, (2\ell)$ for some $a \in \mathbb{N}$ (which is uniquely determined mod $2\ell$). In this case $\gamma \in \gamma_{a,r} + \mathcal{O}_K$ with $\gamma_{a,r} = \frac{a \pm r/\sqrt{p}}{2 \lambda}$ as before. Applying the $U_p$ operator yields 
			\begin{align*} 
			\Big( f(4 \ell \tau) \vartheta_{\chi}(\tau) \Big) \Big| U_p &= 2 \sum_{a \in \mathbb{Z}/2\ell \mathbb{Z}} \sum_{n \in \mathbb{Z} + a^2 / 4\ell} \sum_{r \equiv ab \, (2 \ell)} \overline{\psi_p(\gamma_{a,r})} \chi(r) c(n - r^2 / 4p \ell, \gamma_{a,r}) q^{4 \ell n} \\ &= 2\sum_{a,n,r} \psi_{\ell}(ab) c(n-r^2 / 4p \ell, \gamma_{a,r}) q^{4 \ell n} \\ &= 2\psi_{\ell}(b) \sum_{a,n} \psi_{\ell}(a) \tilde c(n,a) q^{4 \ell n} \\ &= 2\psi_{\ell}(b) \varphi_{\psi_{\ell}} (\Theta F)(\tau).  
			\end{align*}
	\item[(ii)] This is proved similarly to part (i). We do not need to divide by two, since the sum over $r$ runs through only one congruence class (namely, $r \equiv ap \, (2p)$). The definition of $\psi_p$ on $\mathcal{O}_K^{\#}/\mathcal{O}_K$ is such that $\psi_p(\gamma_a) = \psi_p(\frac{a (1+p)}{2\lambda}) = \psi_p(a)$ for all $a \in \mathbb{Z}/2p\mathbb{Z}$. \qedhere
	\end{enumerate}
\end{proof}

If we abbreviate
\[
\Theta_{\chi} G = \Big( G(4\ell \tau) \cdot \vartheta_{\chi}(\tau) \Big) \Big| U_p
\]
for a scalar valued modular form $G$, then the first item of the proposition (the case $\ell \neq p$) can be illustrated by the diagram
\begin{center} \begin{tikzpicture}[node distance=2.5cm, auto]
\node (A) {$S_{k}(-N_{K/\Q})$};
\node (AA) [right of=A] {$ $};
\node (B) [right of=AA] {$S_k^{\epsilon,\AL}(p^{2},\overline{\psi_{p}}\otimes \chi_{p})$};
\node (C) [below of=A] {$S_{k+1/2}((-2\ell))$};
\node (D) [below of=B] {$S_{k+1/2}^{\epsilon,\AL}(\ell^{2},\psi_{\ell})$};
\draw[->] (B) to node {$\Theta_{\chi}$} (D);
\draw[->] (A) to node {$\Theta$} (C);
\draw[->] (A) to node {$\varphi_{\overline{\psi_{p}}}$} (B);
\draw[->] (C) to node {$\varphi_{\psi_{\ell}}$} (D);
\end{tikzpicture}
\end{center}
which commutes up to a constant factor. Note that the horizontal arrows are isomorphisms.

\begin{ex} Let $p = 5$. Fix the element $\lambda = 4+\sqrt{5}$ of norm $\ell = 11$, and fix $b = 7$. We fix the Dirichlet characters $\psi_{11}$ and $\psi_5$ by specifying $\psi_{11}(2) = e^{\pi i / 5}$ and $\psi_5(2) = i$. Up to scalar multiples there is a unique cusp form of (antisymmetric) weight $5$ for the dual Weil representation attached to $(\mathcal{O}_K,-N_{K/\mathbb{Q}})$ with $K = \mathbb{Q}(\sqrt{5})$, and it is 
\begin{align*} 
	F(\tau) &= ( q^{1/5} + 42q^{6/5} - 108q^{11/5} \- 4q^{16/5} - 378q^{21/5} \pm \ldots ) (\mathfrak{e}_{3 / \sqrt{5}} - \mathfrak{e}_{2/\sqrt{5}}) \\ &\quad+ (26q^{4/5} + 39q^{9/5} - 378q^{14/5} + 140q^{19/5} + 420q^{24/5} \pm \ldots) (\mathfrak{e}_{4/\sqrt{5}} - \mathfrak{e}_{1/\sqrt{5}}). 
\end{align*} 
One can compute $F$ using, for example, the algorithm described in \cite{W2} (compare the example of Section 7 there); and after enough coefficients have been computed, one can identify its twisted component sum in $S_5(\Gamma_1(25))$ using standard methods for computing scalar-valued modular forms. The Doi-Naganuma lift of $F$ is, up to a multiple, the well-known product $s_5$ of theta constants for $\mathbb{Q}(\sqrt{5})$ constructed by Gundlach (\cite{Gu}; see also the example of Section 4 of \cite{BB}). The character $\psi_5$ on $\mathcal{O}_K^{\#}/\mathcal{O}_K$ is defined such that e.g. \[
	\psi_5\left(1/\sqrt{5} + \mathcal{O}_K\right) = \psi_5 \left( \frac{1 - 7/\sqrt{5}}{2\lambda} + \mathcal{O}_K \right) = \psi_5(7) = i.
\]
Therefore the twisted component sum of $F$ by $\overline{\psi_{5}}$ is the cusp form 
\[
	\varphi_{\overline{\psi_5}} (F)(\tau) = 2q + 52i q^4 + 84q^6 + 78iq^9 - 216q^{11} - 756iq^{14} - 8q^{16} \pm \ldots \in S_5(\Gamma_0(25),\overline{\psi_5}\otimes \chi_5) = S_5(\Gamma_0(25),\psi_5).
\]
After multiplying 
\begin{align*} 
	\varphi_{\overline{\psi_5}}(F)(44\tau)\vartheta_{\chi}(\tau) &= \Big( 2q^{44} + 52i q^{176} \pm \ldots\Big) \Big( 2q + 2 \zeta_{20}^7 q^{4} - 2\zeta_{20}^9 q^{9} \pm \ldots \Big) \\ &= 4q^{45} + 4\zeta_{20}^7 q^{48} - 4\zeta_{20} q^{53} - 4 \zeta_{20}^4 q^{60} \pm \ldots 
	\end{align*}
and applying $U_{5}$ we get the series 
\[
	4q^9 - 4 \zeta_{20}^4 q^{12} + 4 \zeta_{20}^{18} q^{16}+ 4 \zeta_{20}^2 q^{25} - 104 \zeta_{20}^2 q^{36} \pm \ldots
\]
Dividing by $2 \psi_{11}(b) = -2\zeta_{20}^4$ yields the twisted component sum of the theta contraction $\Theta F$: 
\[
	\varphi_{\psi_{11}}(\Theta F)(\tau) = 2\zeta_{20}^6 q^9 + 2q^{12} + 2 \zeta_{20}^4 q^{16} - 2\zeta_{20}^{18} q^{25} + 52 \zeta_{20}^{18} q^{36} \pm \ldots
\]
From this we can read off the Shimura lift of the underlying vector-valued modular form $\Theta F$: the coefficient of $q^n$ is zero if $11 | n$, and otherwise $\sum_{d | n} \frac{1}{2\psi_{11}(d)} (n/d)^{5-1} c(d^2)$ if $c(n)$ is the coefficient of $q^n$ in $\varphi_{\psi_{11}}(\Theta F)(\tau)$, so 
\begin{align*} s_5(\lambda \tau, \lambda' \tau) 
&= -q^{3} + q^{4} + q^{5} + 10q^{6} - 10q^{8} - 121q^{9} + 98q^{10} + 275q^{12} + 32q^{13} + 140q^{14} \pm \ldots \in S_{10}(\Gamma_0(11)).
\end{align*}
\end{ex}

\begin{ex} Let $p = 13$. Fix the totally positive element $\lambda = \frac{13 + 3\sqrt{13}}{2}$ of norm $\ell = 13$ and fix $b=13$. We fix an odd Dirichlet character $\psi_{13}$ mod $13$ by specifying $\psi_{13}(2) = \zeta_{12} = e^{\pi i / 6}$. The dual Weil representation attached to $(\mathcal{O}_K,-N_{K/\mathbb{Q}})$, $K = \mathbb{Q}(\sqrt{13})$ admits up to scalar multiples a unique cusp form of weight $3$: \begin{align*} F(\tau) &= (q^{1/13} - 33q^{14/13} + 27q^{27/13} + 33q^{40/13} \pm ...) (\mathfrak{e}_{1/\lambda} - \mathfrak{e}_{12/\lambda}) \\ &+ (3q^{3/13} + 5q^{16/13} + 42q^{29/13} - 99q^{42/13} \pm ...) (\mathfrak{e}_{4/\lambda} - \mathfrak{e}_{9/\lambda}) \\ &+ (-7q^{4/13} - 3q^{17/13} - 33q^{30/13} + 49q^{43/13} \pm ...) (\mathfrak{e}_{2/\lambda} - \mathfrak{e}_{11/\lambda}) \\ &+ (0q^{9/13} - 22q^{22/13} + 33q^{35/13} + 15q^{48/13} \pm ...) (\mathfrak{e}_{3/\lambda} - \mathfrak{e}_{10/\lambda}) \\ &+ (11q^{10/13} - 12q^{23/13} + 0q^{36/13} + 50q^{49/13} \pm ...) (\mathfrak{e}_{6/\lambda} - \mathfrak{e}_{7/\lambda}) \\ &+ (21q^{12/13} + 14q^{25/13} - 66q^{38/13} + 9q^{51/13} \pm ...) (\mathfrak{e}_{5/\lambda} - \mathfrak{e}_{8/\lambda}). \end{align*} Under the Doi-Naganuma lift, $F$ is mapped to the cusp form $\omega_3$ used by van der Geer and Zagier to compute the ring of Hilbert modular forms for $\mathcal{O}_K$ (\cite{GZ}, Section 10). The twisted component sum of $F$ by $\psi_{13}$ is $$\varphi_{\psi_{13}}(F)(\tau) = 2\zeta_{12}q + 6\zeta_{12}^3 q^3 - 14\zeta_{12}^2 q^4 - 22q^{10} - 42\zeta_{12}^4 q^{12} \pm ... \in S_3(\Gamma_0(169),\psi_{13} \otimes \chi_{13}).$$ With this we can compute $\omega_3(\lambda \tau, \lambda'\tau)$ as follows: multiply $$\varphi_{\psi_{13}}(F)(52\tau) \cdot \vartheta(\tau) = 2\zeta_{12} q^{52} + 4\zeta_{12} q^{53} + 4\zeta_{12}q^{56} \pm ...$$ and apply the Hecke operator $U_{13}$ to obtain $$\varphi_{\psi_{13}}(\Theta F)(\tau) = 2\zeta_{12} q^4 + 6\zeta_{12}^3 q^{12} - 14\zeta_{12}^2 q^{16} + 4\zeta_{12}q^{17} + 12\zeta_{12}^3 q^{25} \pm ...$$  and therefore the Shimura lift $$\omega_3(\lambda \tau, \lambda' \tau) = q^2 - 3q^4 - 6q^5 + 9q^6 - q^8 + 6q^9 + 57q^{10} \pm ... \in S_6(\Gamma_0(13)).$$
\end{ex}

\bibliographystyle{plainnat}
\bibliofont
\bibliography{\jobname}
\end{document}